\documentclass[12pt,a4paper,withmarginpar]{article}

\usepackage[utf8x]{inputenc}
\usepackage[T1]{fontenc}
\usepackage[a4paper,top=3cm,bottom=3cm,left=3cm,right=3cm,marginparwidth=1.75cm]{geometry}

\usepackage{verbatim}
\usepackage{enumitem}
\input{xy}
\xyoption{all}
\usepackage[all]{xy} 
\usepackage{amsmath}
\usepackage{graphicx}
\usepackage{mathtools}
\usepackage[colorinlistoftodos]{todonotes}
\usepackage[colorlinks=true, allcolors=blue]{hyperref}
\usepackage{amssymb}
\usepackage{amsthm}
\theoremstyle{plain}
\newtheorem{thm}{\textsf{\textbf{Theorem}}}[section]
\newtheorem{lem}[thm]{\textsf{\textbf{Lemma}}}

\newtheorem*{thm*}{\textsf{\textbf{Theorem}}}
\newtheorem*{prop*}{\emph{Proposition}}
\theoremstyle{definition}
\newtheorem{dfn}[thm]{\textbf{\textsf{Definition}}}

\newcommand{\Oee }{\mathcal O}

\newcommand{\nat }{\mathbb N}
\newcommand{\real }{\mathbb R}

\DeclareMathOperator{\inter}{Int}

\DeclareMathOperator{\cl}{Cl}
\newcommand{\fr}{\partial}

\DeclareMathOperator{\id}{id}

\title{On Cartwright-Littlewood Fixed Point Theorem}
\author{Przemysław Kucharski\footnote{AGH University of Science and Technology, Faculty of Applied Mathematics, al. Mickiewicza 30, 30-059 Krak\'{o}w, Poland, e-mail: pkuchars@agh.edu.pl, ORCID iD: \url{https://orcid.org/0000-0002-3826-5827} 
}}

\newcommand{\clu}[1]{\bar{#1}}
\newcommand{\orbp}[1]{\mathcal{O}_{+}(#1)}
\newcommand{\orbm}[1]{\mathcal{O}_{-}(#1)}
\newcommand{\orbpa}[2]{\mathcal{O}_{+}^{#2}(#1)}
\newcommand{\orbma}[2]{\mathcal{O}_{-}^{#2}(#1)}
\newcommand{\plane}[0]{\real^{2}}

\providecommand{\keywords}[1]
{
	\small	
	\textbf{\textit{Keywords---}} #1
}
\providecommand{\declarationofinterest}[1]
{
	\small	
	\textbf{\textit{Declaration of interest---}} #1
}
\begin{document}
		\maketitle
\begin{abstract}
	We prove the following generalization of the Cartwright-Littlewood fixed point theorem. Suppose $ h\colon~\real^{2}\to\real^{2} $ is an orientation preserving planar
	homeomorphism, and $ X $ is an acyclic continuum. Let $ C $ be a component of $ X ∩ h(X) $. If there is a $ c ∈ C $ such that $ \orbp c ⊆ C $ or
	$ \orbm c ⊆ C $ then $ C $ also contains a fixed point of $ h$. Our result also generalizes earlier results of Ostrovski and Boro\'nski, and answers the Question from \cite{boronski_2017}. The proof is inspired by a short proof of the result of  Cartwright and Littlewood due to Hamilton \cite{hamilton_1954}. 
\end{abstract}
\\
\declarationofinterest{Author declares that there are no financial or personal relationships with other people or organizations that could inappropriately influence this work.}
\section{Introduction}
In 1951 Cartwright and Littlewood \cite{cartwright-littlewood} proved that an orientation preserving planar homeomorphism has a fixed point in every acyclic invariant continuum. Short proofs were later given by Hamilton \cite{hamilton_1954}, and Brown \cite{Brown1977ASS}. There is a number of generalizations of the theorem of Cartwright and Littlewood. Bell showed in \cite{Bell1978AFP} that the theorem is true for orientation reversing planar homeomorphisms. His result was further generalized to the plane separating continua by Kuperberg \cite{kuperberg,kuper2} and Boro\'nski \cite{boronski-fixedpointsandperiodicpoints}. More general results, without the assumption that $h$ is a homeomorphism, were obtained in \cite{BlokhMAMS2013}. Yet, we still do not know whether every acyclic planar continuum has the fixed point property \cite[Problem 107]{bookscottish}.

Among generalizations of the Cartwright-Littlewood Theorem there are results of Ostrovski \cite{OSTROVSKI2013915}. We recall three of them, each have its own connection to the result in this paper.
\begin{thm}[{\cite[Thm. 1.1]{OSTROVSKI2013915}}]\label{thmB}
	Let $X\subset \real^{2}$ be a compact, simply connected, locally connected subset of the real plane and let $f\colon X\to Y\subset \real^{2}$ be a homeomorphism isotopic to the identity on $X$. Let $C$ be a connected component of $X\cap Y$. If $f$ has a periodic orbit in $C$, then $f$ also has a fixed point in $C$.
\end{thm}
\begin{thm}[{\cite[Thm. 2.7]{OSTROVSKI2013915}}]\label{thmC}
	Let $D\subset\real^{2}$ be a Jordan domain and $f\colon D\to E\subset \real^{2}$ be an orientation preserving homeomorphism. Let $C$ be a connected component of $D\cap E$. If $f$ has a periodic orbit in $C$, then $f$ also has a fixed point in $C$.
\end{thm}At last, as a corollary to Theorem \ref{thmB}, Ostrovski obtained slightly stronger result, which we are going to use later on, that is
\begin{thm}[{\cite[Cor. 1.2]{OSTROVSKI2013915}}]\label{thmD}
	Let $X\subset\real^{2}$ and $f\colon X\to Y$ be as in Theorem \ref{thmB} and let $C$ be a connected component of $X\cap Y$. If $f$ has no fixed point in $C$, then the orbit of every point $x\in C $ eventually leaves $C$, i.e., there exists $n(x)\in\nat$ such that $f^{n}(x)\notin C$.
\end{thm}

In 2015 Boroński \cite{boronski_2017} developed ideas in \cite{Brown1977ASS} and proved the following
\begin{thm}[{\cite[Theorem 1.1]{boronski_2017}}]\label{boronski}
	Let $h\colon \real^{2}\to\real^{2}$ be an orientation preserving planar homemorphism, and let $C$ be a continuum such that $h^{-1}(C)\cup C$ is acyclic. If there is a point $c\in C$ such that $\orbp{c}\subset C$ or $\orbm{c}\subset C$, then $C$ also contains a fixed point of $h$.
\end{thm}He also asked the following \cite[Question]{boronski_2017}. 

\noindent
{\bfseries Question.} {\itshape Suppose $h : \mathbb{R}^2 \to  \mathbb{R}^2$ is an orientation preserving planar homeomorphism and $X$ is an acyclic continuum. Let $C$ be a component of $X\cap h(X)$. If there is a $c\in C$ such that $\{h^{-i}(c):i\in\mathbb{N}\}\subseteq C$ or $\{h^{-i}(c):i\in\mathbb{N}\}\subseteq C$ must $C$ also contain a fixed point of $h$?}

In this paper we are going to show that we can modify the approach of Hamilton \cite{hamilton_1954} (which in fact was heavily based on the unpublished lemma proved by Newman) to answer Boroński's question in the affirmative. That is, we are going to prove the following result.
\begin{thm}\label{main}
	Suppose $ h\colon~\real^{2}\to\real^{2} $ is an orientation preserving planar
	homeomorphism and $ X $ is an acyclic continuum. Let $ C $ be a component of $ X ∩ h(X) $. If there is a $ c ∈ C $ such that $ \orbp c ⊆ C $ or
	$ \orbm c ⊆ C $ then $ C $ also contains a fixed point of $ h. $
\end{thm} We are also going to exhibit an alternative proof of the above result, by showing that Theorem \ref{thmD} implies Theorem \ref{main}. Interestingly, we can also obtain a version of Theorem \ref{main} with stronger assumption and claim using similar arguments as in the proof of \cite[Thm. 1.2]{boronski_2017} and in the paper \cite{boronski-fixedpointsandperiodicpoints}. Boroński showed in \cite[Thm. 1.2]{boronski_2017} that if a continuum $C$ contains a $k$-periodic orbit $\Oee$ of an orientation reversing planar homeomorphism $g$, for $k>1$, then orbit of every $2-$periodic point linked to $\Oee$ must be contained in $C\cup g^{-1}(C)$, assuming $C\cup g^{-1}(C)$ is acyclic. On the other hand Bonino's results \cite{bonino_2006} implies that always such a $2-$periodic point exists. There is a similar concept of linked periodic orbits for orientation preserving planar or in general surface homomorphisms. Moreover, there exist an orientation preserving counterpart of Bonino's theorem due to Kolev \cite{kolev}. We are going to exploit these facts to construct a proof of Theorem \ref{main}, with the additional assumption that the bounded orbit is periodic. The proof is similar to to the proof of \cite[Thm. 1.2]{boronski_2017}, and highlights a connection between, on one hand side, linking of orbits in terms of appropriate Jordan curves and, on the other hand, co-existance of orbits in continua. 
\begin{thm}\label{main2}
	Suppose $ h\colon~\real^{2}\to\real^{2} $ is an orientation preserving planar
	homeomorphism and $ X $ is an acyclic continuum. Let $ C $ be a component of $ X ∩ h(X) $. If there is a periodic point $ c ∈ C $ such that $ \orbp c ⊆ C $ then $ C $ contains a fixed point of $ h $ linked to $\Oee$.
\end{thm}

\section{Preliminaries}
We follow notation from \cite{boronski_2017}. For a subset $U$ of a topological space we define its closure $\cl U=\clu U$, boundary $\fr U$ and interior $\inter U$ in the usual way. A maximal connected subset of a topological space will be called a connected component, or just a component. We will say that $U$ is forward invariant if it contains its image and strictly invariant if it is equal to its image. Similarly, we define backward invariance. Any sequence $\{q_{n}\}_{n\in\nat}$, where $q_{n}$ are either sets or points, will be abbreviated to $\{q_{n}\}$ if no confusion about the index set is to be made. Respectively, forward and backward orbit of a point $z$ under mapping $h$ will be denoted $\orbpa{z}{h}$ and $\orbma{z}{h}$, and we will sometimes omit superscript and write $\orbp{z}$ and $\orbm{z}$. A subset $X\subset \plane$ will be called plane continuum, if it is compact and connected. Moreover, any continuum that does not separates the plane, that is whose complement is connected, will be called acyclic. It is well known, that plane continuum $X$ is acyclic if and only if we can find a decreasing sequence of closed discs $\{K_{n}\}_{n\in\nat}$ such that $K_{n}\subset \inter K_{n-1}$ and $X=\bigcap_{n\in\nat}K_{n}$. A Jordan curve is a simple, closed curve. We will sometimes identify Jordan curve with its image. A simply connected compact subset of the plane with boundary being a Jordan curve will be called Jordan domain. Any Jordan curve $\gamma$ separates plane into two domains, one of them is bounded and will be called the inner domain of $\gamma$. Any subset of the plane homeomorphic with a closed unit disc will be called a closed disc. We recall a result we will need later on.
\begin{lem}\label{intersection-of-inner-domains}
	Nonempty intersection of finitely many Jordan domains is a disjoint union of Jordan domains.
\end{lem}
Moreover, two continuous mappings $h_{1},h_{2}\colon X\to Y$ of topological spaces $X,Y$ are said to be homotopic in $Y$, if we can find a continuous function $J\colon X\times [0,1]\to Y$, in which case $J$ is called a homotopy. If additionally $x\mapsto J(x,t)$ is a homeomorphism for every $t\in[0,1]$, it is called an isotopy.
In order to give a proof based on the result of B. Kolev, we need to introduce the following
\begin{dfn}\label{unlinked-orbits}
	Let $ \orbp{1}$ and $\orbp{2}$ be two periodic orbits of a homeomorphism $h$ of a surface $S$. We say that $\orbp{1}$ and $\orbp{2}$ are unlinked if there exist two discs $D_{i}\subset S$, $i=1,2$, with the following properties:
	\begin{enumerate}[label=(\arabic*)]
		\item for every $i=1,2$ we have $\Oee_{i}\subset \inter(D_{i})$,
		\item $D_{1}\cap D_{2}=\emptyset$,
		\item for every $i=1,2$ the image $h(\fr D_{i})$ is freely isotopic to $\fr D_{i}$ in $S\setminus(\Oee_{1}\cup\Oee_{2})$.
	\end{enumerate}
\end{dfn}Otherwise the orbits $\Oee_{1}$ and $\Oee_{2}$ are said to be linked. Sometimes we will signify the mapping with respect to which orbits are linked or unlinked, writing $h$-linked or $h$-unlinked.
Two sets that are images of Jordan curves are said to be freely isotopic, if their respective curves are.

B. Kolev formulated his result for $C^{1}$-diffeomorphisms, but, as Bonino pointed out in his introduction \cite{bonino_2006}, it holds true also for homeomorphisms.
\begin{thm}[{\cite{kolev}}]\label{kolev}
	Let $h\colon \plane\to\plane$ be an orientation preserving homeomorphism. Then for every $k$-periodic orbit $\Oee$, for $k>1$, there exists a fixed point linked to $\Oee$.
\end{thm}
The key lemma in proving our results is
\begin{lem}\label{keylemma}
	Let $h\colon \plane\to\plane$ be a homeomorphism and $X$ an acyclic continuum. Let also $\{K_{n}\}_{n\in\nat}$ be a descending sequence of closed discs such that $K_{n}\subset \inter K_{n-1}$ and $X=\bigcap_{n\in\nat}K_{n}$. Then for every connected component $C$ of $X\cap h(X)$, closures $R_{n}$, $n=1,2,...$, of the component of $\inter K_{n}\cap \inter h(K_{n})$ containing $C$ constitute a descending sequence of closed discs with $\bigcap_{n\in\nat}R_{n}=C$.
\end{lem}
\begin{proof}
	Note that $R_{n}$ is a closed disc, since by Lemma \ref{intersection-of-inner-domains} $\inter K_{n}\cap h(\inter K_{n})$ is a disjoint union of inner domains. Moreover, $\{R_{n}\}$ is descending. Indeed, since $ R_{n}$ is connected, it must be contained in some connected component of $\inter K_{n-1}\cap h(\inter K_{n-1})$ and if it was not contained in the component containing $C$, then $C$ would connect two different components in $\inter K_{n-1}\cap h(\inter K_{n-1})$, which is not possible. To see that $C= \bigcap_{n\in\nat}R_{n}$, notice that $\tilde{R}=\bigcap_{n\in\nat}R_{n}$ is connected and $C\subset \tilde{R}\subset X\cap h(X)$. Moreover, $C$ by definition is a maximal connected subset of $X\cap h(X)$, hence $C= \tilde{R}$.
\end{proof}
\begin{lem}\label{lemma-main}
	Let $h\colon\plane\to\plane$ be an orientation preserving homeomorphism, $D\subset\plane$ a closed disc that is neither backward nor forward invariant and $R$ the closure of a component of $h(\inter D)\cap\inter  D\neq\emptyset$. Then there exists an extension $H$ of $h|_{R\cap h^{-1}(R)}$ with no fixed points in $\plane\setminus (R\cap h^{-1}(R))$.
\end{lem}
\begin{proof}
	Put $D'_{1}=\inter D$, $h(D'_{1})=D'_{2}$. By assumption $\fr D'_{1}\cap \fr D'_{2}$ contains at least two points and, by Lemma \ref{intersection-of-inner-domains} the component $G'=\inter R$ of $D'_{1}\cap D'_{2}$ has for its boundary a simple closed curve $J'$. We now transform entire plane, so that the image of $J'$ is a unit circle, more suitable for our argument. Namely, let $f\colon \real ^{2}\to\real ^{2}$ be a homeomorphism, such that $J=f(J')$ is a circle. Put $T=f\circ h \circ f^{-1}$. Then $T$ is orientation preserving. Let also $D_{1}=f(D'_{1})$, $T(D_{1})=(f\circ h\circ f^{-1})(f(D'_{1}))=D_{2}$ and put $G=f(G')$.
	\begin{figure}[h]
		\caption{The set $D_{1}$ is marked by dashed line, whereas $D_{2}$ is marked by continuous line. The unit disc $G$ is marked by alternating dashed and continuous lines, corresponding respectively to arcs $L_{2i}$ and $L_{1i}$.}
		\centering
		\includegraphics[width=17cm]{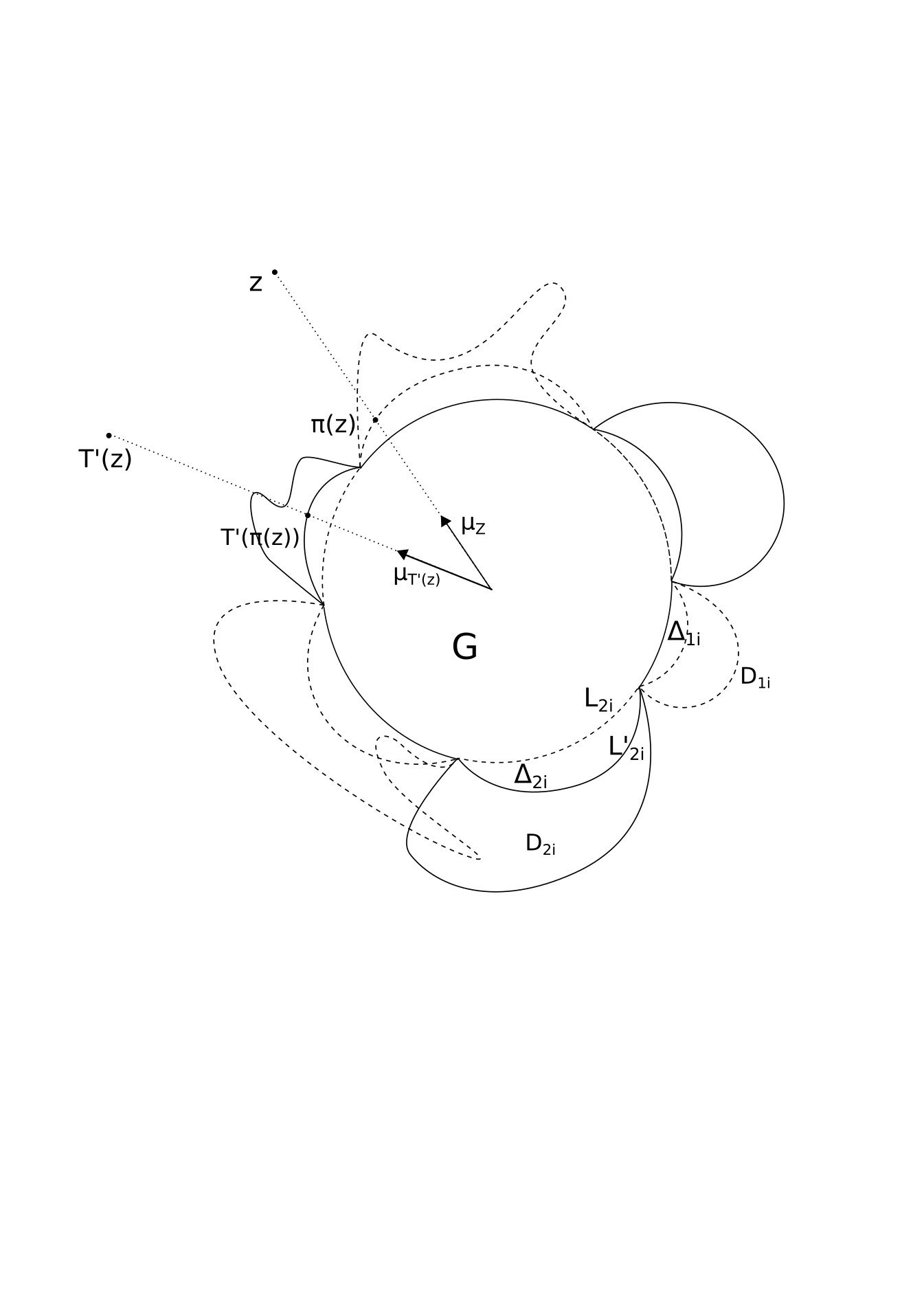}
		\label{fig:drawing}
	\end{figure}
	
	For $r=1,2$ the components $D_{ri}$ of $D_{r}\setminus \clu G$ have each as a boundary a simple closed curve composed of an arc $L_{ri}$ of $J$ and an arc $A_{ri}$ of $\fr D_{r}$ with common endpoints, see Figure \ref{fig:drawing}. For each pair of subscripts $r$ and $i$, let $L'_{ri}$ be a circular arc of radius sufficiently close to $1$ with the same endpoints as $L_{ri}$, such that no two arcs $L'_{ri}$ meet except in endpoints. This is possible since the arcs $L_{ri}$ of $J$ are disjoint except for endpoints. Let $\Delta_{ri}$ be the inner domain of $L_{ri}\cup L'_{ri}$.
	
	Note that we can find isotopies $\kappa_{ri}\colon [0,1]^{2}\to \clu D_{ri}$ and $\iota_{ri}\colon [0,1]^{2}\to \clu\Delta_{ri}$ between respectively $L_{ri}$ and $A_{ri}$, and $L_{ri}$ and $L'_{ri}$ that fixes endpoints. We can also demand that $\clu D_{ri}$ decomposes into disjoint family of curves $\{\kappa_{ri}(s,\cdot)\}_{s
	\in [0,1]}$. Similarly, $\clu \Delta_{ri}$ decomposes into $\{\iota_{ri}(s,\cdot)\}_{s
	\in [0,1]}$. Let us put $\iota_{ri}'(s,t)=\iota_{ri}(\iota_{ri}^{-1}(\kappa_{ri}(s,0))_{1},t)$, where $(\cdot)_{1}$ is a projection on the first coordinate. Then $\iota_{ri}'|_{[0,1]\times\{0\}}=\kappa_{ri}|_{[0,1]\times\{0\}}$. Put $\phi_{ri}(x)=\iota_{ri}'(\kappa_{ri}^{-1}(x))$. It maps $\clu D_{ri}$ onto $ \clu \Delta _{ri} $ and leaves fixed each point of $L_{ri}$. Hence if
	\[\clu\Delta_{r}=\clu G\cup \bigcup_{i}\clu\Delta_{ri}\text{ for } r=1,2 
	\]the functions $\phi_{r}$ defined by
	\begin{align*}
	\phi_{r}|_{\clu G} & =\id\\
	\phi_{r}|_{\clu D_{ri}} & =\phi_{ri}
	\end{align*}are topological maps of $\clu D_{r}$ onto $ \clu\Delta_{r} $ for $r=1,2$.
	
	Let $T'\colon \clu\Delta_{1}\to\clu\Delta_{2}$ be defined as $T'=\phi_{2}\circ T\circ\phi_{1}^{-1}$. Then $T|_{\clu G\cap T^{-1}(\clu G)}=T'|_{\clu G\cap T^{-1}(\clu G)}$. Moreover, if $x\in \clu\Delta_{1}\setminus \clu G$, $x\notin T'(\clu\Delta_{1})$. Similarly, if $x\in \clu G\setminus (\clu G\cap T^{-1}(\clu G))=\clu G\setminus T^{-1}(\clu G)$, then  $T\phi_{1}^{-1}(x)=T(x)\notin \clu G$. As $\phi_{2}$ is a homeomorphism, $\phi_{2}T\phi_{1}^{-1}(x)\notin \clu G$. Hence there are no fixed points of $T'$ outside $\clu G\cap T^{-1}(\clu G)$.
	
	Let $T'$ be extended to the whole of $\real^{2}$ as follows. Let $z$ be a point of $\real^{2}\setminus\clu \Delta_{1}$. Then $z$ is expressible uniquely as $\pi(z)+\rho(z)\mu_{z}$, where $\pi\colon \plane\setminus\Delta_{1}\to\fr \Delta_{1}$ is projection along lines crossing the centre of plane $O$ onto $\fr\Delta_{1}$, $\mu_{z}$ is the unit vector in the direction $Oz$, and $\rho(z)>0$. We define $T'(z)=T'(\pi(z))+\frac{3\rho(z)}{2}\mu_{T'(z)}$. Note that convexity of $\clu\Delta_{1}$ and $\clu\Delta_{2}$ implies that $\pi$ and $\rho$ are well defined and extension $T'$ preserves orientation, see Figure \ref{fig:drawing}. Since for every $z\in\real^{2}\setminus\Delta_{1}$ transformation $\frac{3\rho(z)}{2}\mu_{z}\mapsto\frac{3\rho(z)}{2}\mu_{T'(z)}$ is a bijection, as well as $T'|_{\Delta_{1}}$, we conclude that $T'$ is a homeomorphism of $\real^{2}$. Suppose $T'$ has a fixed point $z=T'(z)\in \plane\setminus \clu \Delta_{1}$. Then the directions from $O$ to $z$ and to $T'(z)$ are the same, hence $\pi(z)=T'(\pi(z))$. Consequently, $\mu_{z}=\mu_{T'(z)}$ and from $T'(\pi(z))+\frac{3\rho(z)}{2}\mu_{z}=\pi(z)+\rho(z)\mu_{T'(z)}$ by subtraction we obtain $\frac{\rho(z)}{2}\mu_{z}=0$, but $\rho(z)>0$, a contradiction. Hence $T'(z)\neq z$, and $T'$ is an orientation preserving homeomorphism of the plane without fixed points in $\plane\setminus \clu G\cap T^{-1}(\clu G)$. We can finally define a bijective orientation preserving extension $H$ of $h|_{\clu G'\cap h^{-1}(\clu G')}$ to the entire plane with no fixed points in $\plane\setminus \clu G'\cap h^{-1}(\clu G')=\plane\setminus (R\cap h^{-1}(R))$ by the formula $H=f^{-1}\circ T'\circ f$.
\end{proof}
\section{Proofs of the fixed point theorem}
As we have mentioned, we will prove Theorem \ref{main} using methods from \cite{hamilton_1954}.
\begin{proof}[First proof of Theorem {\ref{main}}]
	Let $ C $ be a component of $ X ∩ h(X) $ that contains forward or backward orbit of some point. Let $K_{n}$ for $n\in\nat$ be a descending sequence of closed discs, such that $\bigcap_{n\in\nat}K_{n}=X$ and $K_{n}\subset \inter K_{n-1}$. By Lemma \ref{keylemma} there exists a descending sequence $\{R_{n}\}_{n\in\nat}$ of closed discs such that $C= \bigcap_{n\in\nat}R_{n}$, $R_{n}\subset K_{n}\cap h(K_{n})$. If $R_{n}$ is forward or backward invariant, then by Brouwer fixed point theorem we can find a fixed point $w_{n}\in R_{n}$. If $R_{n}$ is neither forward nor backward invariant, then by Lemma \ref{lemma-main} we extend $h|_{R_{n}\cap h^{-1}(R_{n})}$ to $H_{n}$ coinciding with $h$ on $R_{n}\cap h^{-1}(R_{n})$. Then $H_{n}$ will have bounded orbit and by Brouwer Translation Theorem \cite[p. 45, Thm. 8]{Brouwer1912} it must have a fixed point in $R_{n}\cap h^{-1}(R_{n})$. By compactness of components of $X ∩ h(X)$ there exists an accumulation point $w$ of $\{w_{n}\}_{n\in\nat}$, which is also a fixed point of $h$. Note that $w\in \bigcap_{n\in\nat}R_{n}\cap h^{-1}(R_{n})\subset C$. It is a consequence of the fact that $\{w_{k}\}_{k\geq n}\subset R_{n}\cap h^{-1}(R_{n})$ and so $w\in R_{n}\cap h^{-1}(R_{n})$ for every $n\in\nat$.
\end{proof}
Let us now proceed to the proof of \ref{main} that is based on the result of Ostrovski.
\begin{proof}[Second proof of Theorem {\ref{main}}]
	We begin similarly to the proof of \ref{main}. Let $ C $ be a component of $ X ∩ h(X) $ and assume that $\orbp{c}\subset C$ for some $c\in C$. Let $K_{n}$ for $n\in\nat$ be a descending sequence of closed discs, such that $\bigcap_{n\in\nat}K_{n}=X$ and $K_{n}\subset \inter K_{n-1}$. By Lemma \ref{keylemma} there exists a descending sequence $\{R_{n}\}_{n\in\nat}$ of closed discs such that $C= \bigcap_{n\in\nat}R_{n}$ and $R_{n}\subset K_{n}\cap h(K_{n})$. 
	
	We will find a sequence of fixed points $\{w_{n}\}$ such that $w_{n}\in R_{n}$ for every $n\in\nat$. Component $C$ contains a forward orbit, in other words there exists a point, that does not leave $C$. Note that orientation preserving homeomorphism on the closed disc is isotopic to the identity, hence by Theorem \ref{thmD} for every $n\in\nat$ we can find a fixed point $w_{n}\in K_{n+1}\cap h(K_{n+1})\subset \inter K_{n}\cap h(\inter K_{n})$ such that $w_{n}$ is contained in the component $\hat{R}$ of $K_{n+1}\cap h(K_{n+1})$ which contains $C$. Note that $\hat{R}$ will be contained in the component of $\inter K_{n}\cap h(\inter K_{n})$, which contains $C$, that is $R_{n}$, hence $w_{n}\in \hat{R}\subset R_{n}$. Let $w$ be an accumulation point of $\{w_{n}\}$. Note that $w\in\bigcap_{n\in\nat}R_{n}=C$. 
	
	Now, let us consider what happens in the case $\orbm{c}\subset C$ for some $c\in C$. Note that any homeomorphism transforms connected components into connected components. Therefore $h^{-1}(C)$ is a connected component of $X\cap h^{-1}(X)$, which contains $\orbpa{h^{-1}(c)}{h^{-1}}$. We can now use the case of forward orbit to obtain a fixed point $w$ in $h^{-1}(C)$. Then $w=h(w)\in C$.
\end{proof}
At last we present proof of a variation of Theorem \ref{main}, that uses the notion of linked orbits.

\begin{proof}[Proof of the Theorem \ref{main2}]
	Let $h\colon \plane\to\plane$ be an orientation preserving homeomorphism and $X\subset\plane$ an acyclic continuum. Let $K_{n}$ for $n\in\nat$ be a descending sequence of closed discs, such that $\bigcap_{n\in\nat}K_{n}=X$ and $K_{n}\subset \inter K_{n-1}$. Consider $C\subset X\cap h(X)$ a connected component containing a $k$-periodic orbit $\Oee$, for some $k>1$. By Lemma \ref{keylemma} there exists a descending sequence $\{R_{n}\}_{n\in\nat}$ of closed discs such that $C= \bigcap_{n\in\nat}R_{n}$ and $R_{n}\subset K_{n}\cap h(K_{n})$. Moreover, we have also $R_{n}\subset \inter R_{n-1}$. It is a consequence of the inequality $R_{n}\subset \inter K_{n-1}\cap \inter h(K_{n-1})$ and the fact, since $R_{n}$ is connected and $R_{n}\subset R_{n-1}$, that $\inter R_{n-1}$ is a component of $\inter K_{n-1}\cap \inter h(K_{n-1})$. 
	
	We first consider the case when there is a subsequence of $\{K_{n}\}$ that has all elements either backward or forward invariant. Then either $X$ is forward invariant and $C=h(X)$ or $X$ is backward invariant and $C=X$. We will deal only with the case of forward invariance, as the other case is dealt by replacing $h$ with $h^{-1}$. Such an approach is valid, as a point is $h$-fixed if and only if it is $h^{-1}$-fixed and $h$-linked if and only if $h^{-1}$-linked. By Theorem \ref{kolev} we can find a fixed point $w$ linked to $\Oee$, suppose $w\notin K_{1}$. By continuity of $h$ there is $n_{0}\in\nat$ with $ K_{n_{0}}\cup h(K_{n_{0}})\subset \inter K_{1}$. Then $\inter K_{1}\setminus C$ is an open surface and $\fr K_{n_{0}}$, $\fr h(K_{n_{0}})$ are homotopic in $\inter K_{1}\setminus C$. Hence by Theorem 2.1 in \cite{Curves} they are isotopic. Now, $D_{1}=K_{n_{0}}$ and some small enough closed ball $D_{2}$ around $w$ satisfy Definition \ref{unlinked-orbits} of unlinked orbits, contradicting choice of $w$.
	
	Let us now assume that $\{K_{n}\}$ does not contain forward or backward invariant elements. By Lemma \ref{lemma-main} applied to discs $R_{n}$ and $K_{n}$ we obtain homeomorphisms $H_{n}$ with $H_{n}|_{R_{n}\cap h^{-1}(R_{n})}=h|_{R_{n}\cap h^{-1}(R_{n})}$. For every $n\in\nat$, we then use Theorem \ref{kolev}, to find a fixed point $w_{n}\in R_{n}\cap h^{-1}(R_{n})$ of $H_{n}$ linked to $\Oee$. It is necessary that $w_{n}\in R_{n}\cap h^{-1}(R_{n})$, as $H_{n}$ has no fixed points outside $R_{n}\cap h^{-1}(R_{n})$. We can assume that $\{w_{n}\}$ is convergent and $w=\lim_{n\to\infty}w_{n}$. Note that $w\in C$ and, as $H_{n}$ are extensions of $h|_{C\cap h^{-1}(C)}$, $w$ is a fixed point of $h$. We claim that $w$ is linked to $\Oee$. Let us take closed discs $D_{1}$ and $D_{2}$ as in Definition \ref{unlinked-orbits} with $\Oee\subset \inter D_{1}$, $w\in \inter D_{2}$ and empty intersection. By hypothesis, $\fr D_{1}$ is freely isotopic to $h(\fr D_{1})$ in $\plane\setminus(\Oee\cup\{w\})$. Let us denote this isotopy by $J\colon \mathbb S^{1}\times [0,1]\to \plane$. Let $n_{0}\in\nat$ be such that $w_{n_{0}}\notin J(\mathbb S^{1}\times [0,1])\cup D_{1}\cup h(D_{1})$. We will show that $\fr D_{1}$ is freely isotopic to $H_{n_{0}}(\fr D_{1})$ in $\plane\setminus(\Oee\cup\{w_{n_{0}}\})$. This will lead to contradiction, as it is always possible to choose small enough disc $D'_{2}$ around $w_{n_{0}}$ that has boundary freely isotopic to its image under $H_{n_{0}}$ in $\plane\setminus(\Oee\cup\{w_{n_{0}}\})$.
	
	We below present similar argument to \cite[Claim 2.2]{linked}. Let $S=\plane\setminus(\Oee\cup\{w_{n_{0}}\})$. The surface $S$ is open and connected. 
	We will show that $h(\fr D_{1})$ and $H_{n_{0}}(\fr D_{1})$ are homotopic in $S$. Put $P=\inter R_{n_{0}}\cap h^{-1}(\inter R_{n_{0}})$. Note that $h(\fr D_{1})\setminus h(P)$ and $H_{n_{0}}(\fr D_{1})\setminus h(P)$ are disjoint unions of families of arcs without self-intersections. Let us define those families respectively by $\{\alpha_{i}\}_{i\in I\subset \nat}$ and $\{\beta_{i}\}_{i\in I'\subset \nat}$, where $\alpha_{i}$ and $\beta_{i}$ are simple curves. It is possible that $I$ or $I'$ is infinite. As $h|_{\clu P}=H_{n_{0}}|_{\clu P}$, for every $\alpha_{i}$ we can find $\beta_{i'}$ such that they have the same endpoints. Rearranging order of curves, we may assume that the ends of $\alpha_{i}$ and $\beta_{i}$ coincide and $I=I'$. Note that $\alpha_{i}\cup\beta_{i}$ may not be a Jordan curve. Let $F$ be the closure of the union of bounded connected components of $\plane\setminus(\alpha_{i}\cup\beta_{i})$ together with $\alpha_{i}\cup\beta_{i}$. Note that $F$ does not contain any point from $\Oee$, as $\Oee\subset C\cap h^{-1}(C)\subset h(P)$. Moreover, $w_{n_{0}}\notin F$. Indeed, as $F\subset h(D_{1})\cup H_{n_{0}}(D_{1})$, $w_{n_{0}}\in F$ is only possible if $w_{n_{0}}\in h(D_{1})$ or $w_{n_{0}}\in H_{n_{0}}(D_{1})$. As $w_{n_{0}}$ is fixed under $h$ and $H_{n_{0}}$, it would mean that $w_{n_{0}}\in D_{1}$, contradiction with the choice of $w_{n_{0}}$. Consequently, $\alpha_{i}$ is homotopic with $\beta_{i}$ in $F$ by a homotopy that fixes endpoints. Once again, as $h$ and $H_{n_{0}}$ coincide on $\clu P$, $h(\fr D_{1})\cap h(\clu P)=H_{n_{0}}(\fr D_{1})\cap h(\clu P)$ and $h(\fr D_{1})$ and $H_{n_{0}}(\fr D_{1})$ are homotopic in $S$. Consequently, $\fr D_{1}$ is homotopic to $H_{n_{0}}(\fr D_{1})$ in $S$. They are isotopic by Theorem 2.1 in \cite{Curves}, leading to a contradiction with the fact that $\Oee$ and $w_{n_{0}}$ are $H_{n_{0}}$-unlinked.
\end{proof}
\section{Acknowledgments}
This work was supported by the National Science Centre, Poland (NCN), ~grant no. 2019/34/E/ST1/00237. We would like to thank Jan P. Boroński for his constant support and many invaluable comments. 
\bibliography{fixed-point-qstn_v1}

\newcommand{\etalchar}[1]{$^{#1}$}
\begin{thebibliography}{BFM{\etalchar{+}}13}

\bibitem[Bel78]{Bell1978AFP}
H.~Bell.
\newblock {A fixed point theorem for plane homeomorphisms}.
\newblock {\em Fundamenta Mathematicae}, 100:119--128, 1978.

\bibitem[BFM{\etalchar{+}}13]{BlokhMAMS2013}
Alexander~M. Blokh, Robbert~J. Fokkink, John~C. Mayer, Lex~G. Oversteegen, and
  E.~D. Tymchatyn.
\newblock Fixed point theorems for plane continua with applications.
\newblock {\em Memoirs of the American Mathematical Society}, 224(1053):xiv+97,
  2013.

\bibitem[Bon06]{bonino_2006}
Marc Bonino.
\newblock Nielsen theory and linked periodic orbits of homeomorphisms of
  $\mathbb{S}^2$.
\newblock {\em Mathematical Proceedings of the Cambridge Philosophical
  Society}, 140(3):425–430, 2006.

\bibitem[Bor10]{boronski-fixedpointsandperiodicpoints}
Jan~P. Boroński.
\newblock Fixed points and periodic points of orientation-reversing planar
  homeomorphisms.
\newblock {\em Proceedings of The American Mathematical Society},
  138:3717--3722, 10 2010.

\bibitem[Bor17]{boronski_2017}
Jan~P. Boroński.
\newblock {On a generalization of the Cartwright–Littlewood fixed point
  theorem for planar homeomorphisms}.
\newblock {\em Ergodic Theory and Dynamical Systems}, 37(6):1815–1824, 2017.

\bibitem[Bor21]{linked}
J.~P. Boroński.
\newblock {Linked orbits of homeomorphisms of the plane and Gambaudo-Kolev
  Theorem}.
\newblock 2021.

\bibitem[Bro12]{Brouwer1912}
L.~E.~J. Brouwer.
\newblock Beweis des ebenen translationssatzes. (mit 9 figuren im text).
\newblock {\em Mathematische Annalen}, 72:37--54, 1912.

\bibitem[Bro77]{Brown1977ASS}
Morton Brown.
\newblock {A short short proof of the Cartwright-Littlewood theorem}.
\newblock {\em {Proceedings of The American Mathematical Society}}, 65(2):372,
  1977.

\bibitem[Car51]{cartwright-littlewood}
Littlewood Cartwright.
\newblock {Some fixed point theorems. With appendix by H. D. Ursell}.
\newblock {\em Ann. of Math.}, 54(2):1–37, 1951.

\bibitem[Eps66]{Curves}
D.~B.~A. Epstein.
\newblock {Curves on 2-manifolds and isotopies}.
\newblock {\em Acta Mathematica}, 115(none):83 -- 107, 1966.

\bibitem[Ham54]{hamilton_1954}
O.~H. Hamilton.
\newblock {A Short Proof of the Cartwright-Littlewood Fixed Point Theorem}.
\newblock {\em Canadian Journal of Mathematics}, 6:522–524, 1954.

\bibitem[Kol90]{kolev}
Boris Kolev.
\newblock {Point fixe li{\'e} {\`a} une orbite p{\'e}riodique d'un
  diff{\'e}omorphisme de $\plane$}.
\newblock {\em {Comptes rendus de l'Acad{\'e}mie des sciences. S{\'e}rie I,
  Math{\'e}matique}}, 310(12):831--833, 1990.
\newblock in French, 4 pages.

\bibitem[Kup91a]{kuper2}
Krystyna Kuperberg.
\newblock Fixed points of orientation reversing homeomorphisms of the plane.
\newblock {\em Proceedings of The American Mathematical Society},
  112(1):223–229, 1991.

\bibitem[Kup91b]{kuperberg}
Krystyna Kuperberg.
\newblock A lower bound for the number of fixed points of orientation reversing
  homeomorphisms.
\newblock In {\em The geometry of {H}amiltonian systems ({B}erkeley, {CA},
  1989)}, volume~22 of {\em Math. Sci. Res. Inst. Publ.}, pages 367--371.
  Springer, New York, 1991.

\bibitem[Mau15]{bookscottish}
R.~D. Mauldin.
\newblock {\em {The scottish book: Mathematics from the scottish café, with
  selected problems from the new Scottish book, second edition}}.
\newblock 01 2015.

\bibitem[Ost13]{OSTROVSKI2013915}
Georg Ostrovski.
\newblock Fixed point theorem for non-self maps of regions in the plane.
\newblock {\em Topology and its Applications}, 160(7):915--923, 2013.

\end{thebibliography}
\bibliographystyle{alpha}

\end{document}